\documentclass{amsart}[12pt]
\usepackage{amsmath, amsthm, amssymb, enumerate, amsfonts, mathrsfs, xypic,pxfonts, color}

\xyoption{all}

\definecolor{red}{rgb}{1,0,0}
\definecolor{green}{rgb}{0,1,0}
\definecolor{blue}{rgb}{0,0,1}

\newtheorem{definition}{Definition}
\newtheorem{theorem}{Theorem}

\newtheorem{lemma}{Lemma}

\newcommand{\scri}{\mathscr{I}}
\newcommand{\op}{\operatorname}

\newcommand{\Aii}[2]{{\textstyle #1\!\!\!\!-\!\!\!-\!\!\hspace{-0.3pt}#2}}
\newcommand{\Aiii}[3]{{\textstyle #1\!\!\!\!-\!\!\!-\!\!\!\!#2\!\!\!\!-\!\!\!-\!\!\!\!#3}}
\newcommand{\Aooo}{\Aiii{\bullet}{\bullet}{\bullet}}
\newcommand{\Aoox}{\Aiii{\bullet}{\bullet}{\times}}
\newcommand{\Aoxo}{\Aiii{\bullet}{\times}{\bullet}}
\newcommand{\Aoxx}{\Aiii{\bullet}{\times}{\times}}
\newcommand{\Axoo}{\Aiii{\times}{\bullet}{\bullet}}
\newcommand{\Axox}{\Aiii{\times}{\bullet}{\times}}
\newcommand{\Axxo}{\Aiii{\times}{\times}{\bullet}}
\newcommand{\Axxx}{\Aiii{\times}{\times}{\times}}
\newcommand{\Aoo}{\Aii{\bullet}{\bullet}}
\newcommand{\Axo}{\Aii{\times}{\bullet}}
\newcommand{\Aox}{\Aii{\bullet}{\times}}
\newcommand{\Axx}{\Aii{\times}{\times}}

\newcommand{\osarray}[2]{\overset{\displaystyle{\overset{\displaystyle{#1}}{#2}}}{\tfrac{}{}}}

\newcommand{\dynkinexample}{
$$
\def\objectstyle{\displaystyle}
\xymatrix{
&&&&&\bullet\\
\osarray{a}{\bullet}\ar@{-}[r]&\bullet\ar@{-}[r]&\bullet\ar@{..}[r]&\bullet\ar@{-}[r]&\bullet\ar@{-}[ur]\ar@{-}[dr]&\\
&&&&&\bullet
}
$$
}

\newcommand{\comment}[1]{}

\title[Asymptotically shearfree congruences]{Asymptotically shearfree congruences in $(2,2)$ spacetimes and Burgers' equation}
\author{Jonathan Holland}
\address{Department of Mathematics, 
University of Pittsburgh}
\author{George Sparling}

\begin{document}
\begin{abstract}
The paper proves that any asymptotically shearfree congruence at the conformal infinity $\scri$ in a $(2,2)$-signature spacetime is determined locally by a solution to the pair of forced inviscid Burgers' equations $L_u+LL_x=\sigma(u,x,y,L)$ and $M_u+MM_y=\tilde{\sigma}(u,x,y,M)$ where $u,x,y$ are Bondi coordinates of $\scri$.  The functions $\sigma$ and $\tilde{\sigma}$ are determined naturally by the projective structure on the $\alpha$ and $\beta$ surfaces that foliate $\scri$.
\end{abstract}
\maketitle

\nocite{*}

\section{Introduction}
The purpose of this article is to study asymptotically shearfree congruences of null geodesics in ultrahyperbolic ($(2,2)$-signature) spacetimes in terms of structure at infinity.  A congruence is a foliation of spacetime by null geodesics.  It is shearfree if (roughly) a beam of light does not distort from a circular profile to an elliptical one \cite{Sachs}.  Shearfree congruences in Lorentzian ($(1,3)$-signature) spacetimes have a long history of importance in relativity theory, because they serve as characteristics for algebraically special solutions of field equations in spacetime, including the Einstein vacuum equations (\cite{RobinsonTrautman1960}, \cite{RobinsonTrautman1962}, \cite{KerrSchild}, \cite{Sommers}).  A detailed history of the role of shearfree congruences can be found in Trautman \cite{Trautman}.

We here study the problem of describing an (asymptotically) shearfree congruence in terms of a minimal set of initial data at conformal infinity.  The conformal infinity of a spacetime is a null cone known as $\scri$ or ``scri''.  An asymptotically shearfree congruence is specified by its scattering data at $\scri$.  In turn, this scattering data defines an embedding of $\scri$ into the space of unparametrized null geodesics (called projective null twistor space, or $\mathbb{PN}$).  In spacetimes of Lorentzian signature, the space of null geodesics carries a CR structure, and the shearfree condition is equivalent to the embedding being an embedding in the category of CR manifolds; this result is known as  Kerr's theorem (\cite{PenroseAlg}; see \cite{HollandSparling} for a proof of the theorem in this general form).  In the case of $(2,2)$ signature spacetimes, the space of null geodesics carries in a natural manner a real analog of a CR structure, and the shearfree condition is equivalent to the condition that this embedding be a real CR embedding.  

A crossection of $\scri$ is canonically identified with an indefinite quadric in projective $3$-space, and so is foliated by two families of lines.  Thus it can be identified with $\mathbb{P}^1\times\mathbb{P}^1$, and so $\scri$ itself is a trivial line bundle over $\mathbb{P}^1\times\mathbb{P}^1$ that is foliated by two families of two-planes, so-called the $\alpha$-planes and $\beta$-planes.  When the spacetime is flat, the embedding of $\scri$ into the space of null geodesics is determined by a pair of functions $L$ and $M$ in the $\alpha$-planes and $\beta$-planes (respectively), satisfying a pair of independent Burgers' \cite{Burgers} equations:
\begin{equation}\label{BurgersEquations}
L_u+LL_x = 0\qquad M_u+MM_y=0.
\end{equation}
The initial value problem for Burgers' equation is well-understood locally (where shocks are forbidden); see, for instance, \cite{CourantHilbert}.  The characterstics are lines, and the solution is constant along these lines (and equal to the slope of that line).  So a solution is locally determined by specifying initial Cauchy data on a cross section of the line bundle $\scri\to\mathbb{P}^1\times\mathbb{P}^1$.  Summarizing,

\begin{theorem}\label{localexistence}
Let $S$ be a crossection of $\scri$ and $\kappa_S:S\to\mathbb{PN}$ a section of the natural fibration of $\mathbb{PN}$ over $\scri$ that is transverse to the CR structure.  Then there is an open neighborhood $U$ of $S$ in spacetime and a shearfree congruence in $U$ that agrees with $\kappa_S$ on $S$.
\end{theorem}

Here the condition of transversality is that the restriction of $\kappa_S$ to each $\mathbb{P}^1$ on $S$ must avoid a tangent direction defined by the value of $\kappa_S$.  In general, Cauchy data $\kappa_S:S\cong \mathbb{P}^1\times\mathbb{P}^1\to \mathbb{PN}$ is specified by a pair of functions of two variables.  But for a global solution, Burgers' equation does not have non-constant solutions.  So in that case the Cauchy data $\kappa_S$ reduces to a pair of functions of {\em one} variable.

When the spacetime is {\em curved}, the equations \eqref{BurgersEquations} are deformed to a pair of {\em inhomogeneous} Burgers' equations by forcing terms $\sigma=\sigma(u,x,y,L)$ and $\tilde{\sigma}=\tilde{\sigma}(u,x,y,M)$:
\begin{equation}\label{GeneralBurgers}
L_u+LL_x = \sigma,\qquad M_u+MM_y = \tilde{\sigma}.
\end{equation}
General equations of this kind are considered in \cite{CourantHilbert}.  The characteristics of these equations are defined by second order ordinary differential equations.  The forcing terms $\sigma$ and $\tilde{\sigma}$ are at most cubic in $L$ and $M$, respectively, and the characteristics are the geodesics associated to a projective structure in the $\alpha$ and $\beta$ planes of $\scri$.

This article begins with a discussion of Burgers' equation.  Solutions of Burgers' equation are constant along lines, and the value of the solution is equal to the slope of that line.  This can be geometrized to form the notion of a {\em Burgers' function}, which is a function having precisely this property.  This is defined formally in terms of homogeneous spaces for the (projective) general linear group.  In this paper, we are interested in solutions of Burgers' equation {\em without shocks}, and this is considerably easier to analyze than the case with shocks.  However, the geometric approach can accommodate shock-like behavior by defining Burgers' functions for locally finite sheaves over the base space.  There is scope for such a generalization in the analysis of generalizations of congruences, such as Kerr congruences which are generically multiply-valued.

Because they are needed for the analysis of curved spacetime, Burgers' functions are then generalized to functions from a two-dimensional space equipped with a notion of geodesics (in the form of a vector field on the cotangent bundle) to the space of geodesics.  Working locally, there is a close connection with the theory of second-order ordinary differential equations modulo point equivalence.  Such a Burgers' function maps (locally) from a dynamical space into a space of solutions of a second order differential equation, with the property that the function be constant along each solution that it defines.  There is a dual Burgers' function that maps the other way, and the second order equation that mediates the dual Burgers' function is the dual second order differential equation of \'{E}lie Cartan \cite{CartanProjective} (see, for instance, \cite{CrampinSaunders} and the references therein for a modern account of duality of second order equations).  These two Burgers' functions are completely characterized by a caustic curve in the microtwistor space.

The article then defines the flat ultrahyperbolic spacetime and its conformal compactification $\mathbb{K}$ obtained by adding a null cone $\scri$ at infinity.  This space is called {\em Klein space}, and it is identified with the Klein quadric: the null quadric in the projective space of $\wedge^2\mathbb{T}$ where $\mathbb{T}$ is a $4$-dimensional real vector space.  The space $\mathbb{T}$ is the {\em twistor space} of $\mathbb{K}$.  All of the flag manifolds associated to $\mathbb{T}$ come into play at this point.  The article then draws the link between shearfree congruences in $\mathbb{K}$ and Burgers' functions at $\scri$.

The next section of the article is devoted to studying the curved case.  In that case, the spacetime $M$ carries a pair of distributions $\mathbf{D}$ and $\mathbf{D}'$ of $2$-planes whose integral manifolds (if they exist) are the $\alpha$ and $\beta$ surfaces in $M$.  We shall confine attention to the case where $\scri$ is a {\em shearfree hypersurface}, meaning that it is foliated by $\alpha$-surfaces and by $\beta$-surfaces.  Under suitable falloff conditions at infinity, this is automatically true for solutions of the Einstein vacuum equation.  The restriction of $\mathbf{D}'$ to each $\alpha$-surface in the foliation gives a projective structure on the surface, and likewise with the restriction of $\mathbf{D}$ to each $\beta$-surface.  The geodesics of these projective structures are the null geodesics of the conformal structure.  The analog \eqref{GeneralBurgers} of Burgers' equation can be formulated in this more general setting, with the forcing terms $\sigma$ and $\tilde{\sigma}$ being cubic polynomials defined in terms of the connection coefficients of the projective structure.

For curved spacetimes, shearfree congruences do not exist in general, and it is necessary to study {\em asymptotically shearfree congruences}: congruences of null geodesics whose shear vanishes on $\scri$.  The local result from the flat case then remains true in the case of curved spacetimes.

\section{Real CR manifolds}\label{RealCR}
A basic object of study is that of a {\em real CR manifold}.  This is a smooth manifold $X$ of dimension $2n+1$ together with a pair of smooth distributions $D$ and $D'$ of $n$-planes on $X$ such that:
\begin{itemize}
\item $D$ and $D'$ are integrable in the sense of Frobenius
\item $D\cap D'=0$
\end{itemize}

The Levi form associated to a real CR manifold is the bilinear form with values in the line bundle $TM/D\oplus D'$:
$$L(X,Y) = [X,Y] \pmod{D\oplus D'}.$$
By integrability, $D$ and $D'$ are isotropic subspaces for $L$.  The CR structure is called non-degenerate if $L$ is a non-degenerate pairing.  Non-degeneracy is equivalent to the distribution of $2n$-planes $D\oplus D'$ being a contact structure on $M$.

A non-degenerate real CR manifold therefore consists of a $(2n+1)$-dimensional contact manifold together with a pair of independent Legendrian distributions $D$ and $D'$, each of which is integrable.

{\em Example.}  Let $A$ be a real $2$-dimensional manifold and let $X=\mathbb{P}T^*A$ be the projective cotangent bundle.  Equip the cotangent bundle $T^*A$ with the canonical one-form $\theta$ and let $H$ be the vector field generating the scaling in the fibers.  Since $\theta(H)=0$ and $\mathscr{L}_H\theta=\theta$, the distribution of $3$-planes $\theta=0$ descends to a distribution of $2$-planes on $X$, a contact structure.  Let $D'$ be the tangent distribution to the fibers of the projection $X\to A$.  Let $D$ be a nonvanishing contact vector field on $X$ that is transverse to the fibers of $X\to A$.  The integral curves of $D$ are called {\em geodesics}.

Conversely, if $X$ is a non-degenerate real CR $3$-manifold, then around any point of $X$ there is a neighborhood that is contactomorphic to an open subset of $\mathbb{P}T\mathbb{R}^2$, the projective cotangent bundle of $\mathbb{R}^2$ with its canonical contact structure.  Moreover, we can ensure that in this neighborhood the distribution $D'$ is tangent to the fibers of the projection $\mathbb{P}T\mathbb{R}^2\to\mathbb{R}^2$.  This identification is unique up to local diffeomorphism of $\mathbb{R}^2$.

\subsection{Second order ODEs}\label{ODES} Locally, a non-degenerate real CR manifold $X$ reduces to an open subset of $\mathbb{P}T^*\mathbb{R}^2$ together with its canonical contact structure, and equipped with a direction vector field $D$.  Let $(u,x)$ be coordinates on $\mathbb{R}^2$ and introduce a local fiber coordinate $p$ on $\mathbb{P}T^*\mathbb{R}^2$ so that the canonical contact structure has representative $\theta=dx-p\,du$.  Then a representative of the direction field $D$ is given by
$$D = \frac{\partial}{\partial u} + p\frac{\partial}{\partial x} + \sigma(u,x,p)\frac{\partial}{\partial p}$$
up to an overall multiple, where $\sigma$ is a smooth function of three variables.  An integral curve of this vector field projects down to the $xy$-plane to a solution of the second order differential equation $x''=\sigma(u,x,x')$.

Conversely, suppose that $G(u,x;a,b)=0$ is (locally) a complete integral of the ODE $x''=\sigma(u,x,x')$.  Let $\mathbb{R}^4$ be the space with variables $(u,x,a,b)$.  This is equipped with two projections onto $\mathbb{R}^2$: $\pi_1:(u,x,a,b)\mapsto (u,x)$ and $\pi_2:(u,x,a,b)\mapsto (a,b)$.  Let $X$ be the hypersurface $G(u,x;a,b)=0$ in $\mathbb{R}^4$.  Then $D'=\ker d\pi_1$ and $D=\ker d\pi_2$ define two distributions on $X$ that satisfy the conditions for a real CR structure.  The quotient by one distribution maps into the space of dynamical variables $(u,x)$, and the other maps into the space of solutions $(a,b)$ of the ODE.  Hence $X$ is locally isomorphic to the real CR manifold of the previous paragraph.

By interchanging the roles of the two sets of variables $(u,x)$ and $(a,b)$, instead regarding $(a,b)$ as the dynamical variables, there results another second order differential equation.  The point equivalence class of this second order differential equation depends only on the point equivalence class of the original second order equation, and is called the {\em dual equation}.  

\section{Burgers' functions}\label{BurgersFunctions}
The projective plane $\mathbb{P}^2$ is the space of $1$-dimensional flags $V_1\subset\mathbb{R}^3$, and the dual projective denoted by $(\mathbb{P}^2)^*$ is the space of lines in $\mathbb{P}^2$ or, equivalently, the space of $2$-dimensional flags $V_2\subset\mathbb{R}^3$.  These are both homogeneous spaces for the general linear group $\op{GL}(3,\mathbb{R})$ and can be obtained as a quotient of $\op{GL}(3,\mathbb{R})$ by an appropriate parabolic (upper-triangular) subgroup.  We employ the notation of \cite{BastonEastwood} in labelling the parabolic subgroups using Dynkin diagrams with nodes crossed unless the negative of the corresponding positive root participates in the parabolic subgroup.  In particular, $P(\Axx)$ is the Borel subgroup and $\op{GL}(3,\mathbb{R})=P(\Aoo)$.  

The correspondence space between $\mathbb{P}^2$ consists of all lines in $\mathbb{P}^2$ together with a point on that line.  Equivalently, this is the space of flags $V_1\subset V_2\subset\mathbb{R}^3$ with $\dim V_i=i$.  This is the quotient of $\op{GL}(3,\mathbb{R})$ by the Borel group.  The lattice of parabolic homogeneous spaces of $\op{GL}(3)$, labelled according to their Dynkin type, is:
$$
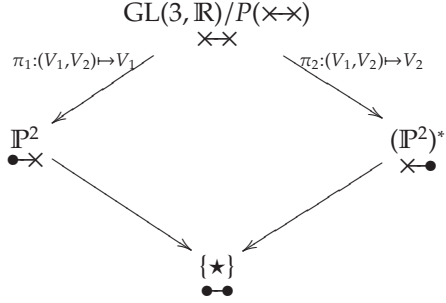
\xymatrix{
&\underset{\Axx}{\op{GL}(3,\mathbb{R})/P(\Axx)}\ar[dl]_{\pi_1:(V_1,V_2)\mapsto V_1}\ar[dr]^{\pi_2:(V_1,V_2)\mapsto V_2}&\\
\underset{\Aox}{\mathbb{P}^2}\ar[dr] && \underset{\Axo}{(\mathbb{P}^2)^*}\ar[dl]\\
&\underset{\Aoo}{\{\star\}}&
}
$$
The mappings $\pi_1$ and $\pi_2$ define a pair of foliations of the $3$-dimensional space $\op{GL}(3,\mathbb{R})/P(\Axx)$ by curves.  The tangents to these curves $\ker d\pi_1=D$ and $\ker d\pi_2=D'$ define a real CR structure.  This CR structure is non-degenerate.  In fact, $\op{GL}(3,\mathbb{R})/P(\Axx)$ can be naturally identified with the projective cotangent bundle of $\mathbb{P}^2$, and $(\mathbb{P}^2)^*$ is the quotient by the geodesic flow.

\begin{definition}\label{DefBurgersFunction}
Let $U\subset \mathbb{P}^2$ be an open set.  A {\em Burgers' function} on $U$ is a map $f:U\to (\mathbb{RP}^2)^*$ such that the the preimage of any line $\ell$ in the range of $f$ is $\ell\cap U$.
\end{definition}

To describe a $C^1$ Burgers' function explicitly, introduce rectangular coordinates $(u,x)$ in an affine patch of $\mathbb{RP}^2$ centered near a point of $U$.  Coordinatize the open subset of $\mathbb{L}$ that excludes lines parallel to the $x$-axis with pairs $(X,L)$ where $X$ is the point of intersection of a line with the $x$ axis and $L$ is its slope.  A Burgers' function $(X(x,u),L(x,u))$ must then satisfy the equations
\begin{align*}
L(x+uL_0(x), u) &= L_0(x)\\
X(x+uL_0(x), u) &= x
\end{align*}
Differentiating the first of these equations with respect to $u$ gives Burgers' equation
\begin{equation}\label{burgers1}
L_u + LL_x = 0,\qquad L(x,0)=L_0(x).
\end{equation}
Differentiating the second gives the transport equation
$$X_u + LX_x=0,\qquad X(x,0)=x.$$

Suppose now that a $C^1$ solution to equation \eqref{burgers1} is given.  Then this solution must be constant on the line $u\mapsto (x+uL_0(x),u)$, and therefore is a Burgers' function.  Since the transport equation is uniquely solvable with $C^1$ solution and involves no freedom in the choice of initial conditions, it follows that Burgers' functions correspond to solutions of Burgers' equation:

\begin{lemma}
Every $C^1$ Burgers' function $f:U\to(\mathbb{RP}^2)^*$ is a solution of Burgers' equation in a local rectangular coordinate chart.  Conversely, any $C^1$ function that solves Burgers' equation locally in a rectangular coordinate chart about every point of $U$ is a Burgers' function.
\end{lemma}

\subsection{Cauchy problem}
Let $I=[0,1]$.  The Cauchy problem for a Burgers' function is to prescribe values on a $C^1$ embedded curve $\gamma: I \to \mathbb{P}^2$.  That is, for a function $L_\gamma : I \to (\mathbb{P}^2)^*$ to find an open subset $U$ of $\mathbb{P}^2$ containing $\gamma(I)$ and a Burgers' function $L:U\to(\mathbb{P}^2)^*$ such that $L_\gamma=L\circ\gamma$.  Define the tangent map $T_\gamma:I\to (\mathbb{P}^2)^*$ to be the function that associates to $\gamma$ its tangent line at each point.  

\begin{lemma}\label{BurgersCauchy}
A sufficient condition for a local solution of the Cauchy problem is that $T_\gamma(x)\not=L_\gamma(x)$ for all $x\in I$.
\end{lemma}

\begin{proof}
Consider the graph of the solution in $\mathbb{P}^2\times(\mathbb{P}^2)^*$.  Let $\Gamma(x) = (\gamma(x),L_\gamma(x))$ for $x\in I$.  The graph of a solution $L$ is a surface lying in the quadric hypersurface $Q=\{(x,y)\in\mathbb{P}^2\times (\mathbb{P}^2)^*\mid \langle x,y\rangle =0\}$ and contains the curve $\Gamma(I)$.  The condition that  $T_\gamma(x)\not=L_\gamma(x)$ is equivalent to $\Gamma$ being transverse to the fibers of $Q\to (\mathbb{P}^2)^*$, which gives sufficiency as a consequence of the rank theorem.
\end{proof}

We note for later use that the only (globally defined) Burgers' function on an affine patch $\mathbb{R}^2\subset\mathbb{RP}^2$ must be constant, since two lines of unequal slopes must intersect.  Burgers' functions do not allow the possibility of shocks that are traditionally associated with discontinuous weak solutions of Burgers' equation.

The definition of Burgers' functions can be generalized to allow multiple-valued functions.  The precise generalization depends on what kinds of singularities are allowed.  The following is one such generalization so that the Burgers' function is a multiple-valued smooth function locally almost everywhere:

\begin{definition}
A {\em Burgers' surface} is a two-dimensional immersed submanifold $i:B\to\op{GL}(3)/P(\Axx)$, with $i$ a smooth immersion, such that $\op{rank}\pi_2\circ i=1$.
\end{definition}

As a consequence of Sard's theorem together with Fubini's theorem, $\op{rank}\pi_1\circ i=2$ almost everywhere on $B$.  In a neighborhood $V$ of any point of $B$ where $\op{rank}\pi_1\circ i=2$, there is locally a Burgers' function $f:\pi_1(V)\to \mathbb{RP}^2$ such that $V\subset \pi_2^{-1}f\pi_1(V)$.  The definition of a Burgers' surface allows Burgers' functions that are ramified over a set of points in $\mathbb{RP}^2$, as illustrated in Example \ref{example}.

A Burgers' surface is called {\em maximal} if it contains each of the fibers over $\pi_2$: $i(B)=\pi_2^{-2}\pi_2i(B)$.  Any Burgers' surface is contained in a maximal Burgers' surface.  A connected maximal Burgers' surface is a cylinder (with $S^1$ fibers) over an immersed curve  $\gamma:I\to(\mathbb{P}^2)^*$.  This immersed curve is equipped with a canonical lift, the {\em caustic curve} $\overline{\gamma}:I\to \op{GL}(3)/P(\Axx)$ defined by associating at each point $x\in I$ the pair $(x,L_\gamma(x))$ where $L_\gamma(x)$ denotes the line tangent to $\gamma$ at $x$.  The caustic curve is the set of points where $i:TB\to D\oplus D'$ maps onto the contact distribution.  In particular, it is a Legendrian curve.  

\subsection{Example}\label{example}
Let $(x,y,z)\in\mathbb{R}^3$ be projective coordinates on $\mathbb{P}^2$ and $(p,q,r)$ dual coordinates on $(\mathbb{P}^2)^*$.  The space $\op{GL}(3)/P(\Axx)$ can be identified with the hypersurface $xp+yq+zr=0$ in $\mathbb{P}^2\times(\mathbb{P}^2)^*$.

Let $Q$ be the circle in $\mathbb{P}^2$ given by the equation $x^2 + y^2 = z^2$.  Consider the Burgers' surface whose fibers over each point outside $Q$ are the pair of tangent lines to the circle through $Q$ and whose fiber over each point of $Q$ is the tangent line through the point.  Opposite tangent lines to the circle intersect in the line at infinity, so the two sheets of the Burgers' surface ramify at infinity.

The caustic curve is the locus of the equations
$$x^2+y^2=z^2,\qquad xp+yq+zr=0,\qquad p^2+q^2=r^2.$$

Examples of Cauchy curves include curves of the form $x^2+y^2=az^2$ for $a>1$, since these are everywhere transverse to the tangent lines of the circe.

\subsection{Generalized Burgers' surfaces}
The general setup is as follows.  Let $(X,D,D')$ be a three-dimensional real CR manifold and $C$ a connected smoothly embedded Legendrian curve in $X$ that is transverse to $D'$.  A maximal connected integral manifold $B$ of $D$ that contains $C$ is called a {\em generalized Burgers' surface}.  In this setting the curve $C$ is the associated Burgers' curve.  Working locally, $X$ is isomorphic to the projective cotangent bundle of a surface $A$ with $D'$ the vertical distribution of $\pi:X\to A$.  In a neighborhood $U$ of a generic point of $B$, $\pi|_U : U\to A$ is a local diffeomorphism.  Let $V=\pi(U)$.  Then there is a section $L : V\to B$ of $\pi$. At each point $x\in \pi(U)$, $L(x)$ defines a geodesic through $x$ such that whenever $x'$ is another point of that geodesic in $V$, $L(x)=L(x')$. Thus the definition of a generalized Burgers' surface generalizes that of a Burgers' function in Definition \ref{DefBurgersFunction}.

As in Section \ref{ODES}, represent $A$ as an open subset of the real $ux$-plane and $D$ by the vector field
$$D = \frac{\partial}{\partial u} + p\frac{\partial}{\partial x} + \sigma \frac{\partial}{\partial p}.$$
In these coordinates, the section $L$ giving the generalized Burgers' surface is described by a value of $p$ at each point of $A$.  Moreover, $L$ satisfies the inhomogeneous Burgers' equation
$$L_u + LL_x = \sigma(u,x,L).$$

Note that interchanging the roles of $D$ and $D'$ in the definition of a generalized Burgers' surface gives a different surface $B'$ such that $B\cap B'$ contains the Legendrian curve $C$.  Both of these surfaces are determined by $C$, and they also determine $C$.  These two Burgers' surfaces are connected by the same duality that connects their associated characteristic curves.  In particular, the function $\sigma$ associated to the dual Burgers' surface is the function corresponding to the dual second order differential equation.  The Burgers' surface and its dual each give rise to a caustic, the pair of which defines a Legendrian link in the contact manifold $X$.

\section{Klein space}
Let $\mathbb{T}$ be a fixed four-dimensional real vector space.  Define the following flag manifolds:
\begin{itemize}
\item Projective twistor space: $\mathbb{PT}=\{V_1\subset\mathbb{T}\mid \dim V_1=1\}$
\item Dual projective twistor space: $\mathbb{PT}^*=\{V_3\subset\mathbb{T}\mid \dim V_3=3\}$
\item Projective spin bundle: $\mathbb{PS}=\{V_1\subset V_2\subset \mathbb{T}\mid \dim V_1=1,\ \dim V_2=2\}$
\item Projective null twistor space: $\mathbb{PN}=\{(V_1,V_3)\mid V_1\subset V_3\subset\mathbb{T},\ \ \dim V_1=1,\ \dim V_3=3\}$
\item Primed projective spin bundle: $\mathbb{PS}'=\{V_2\subset V_3\subset \mathbb{T}\mid \dim V_2=2,\ \dim V_3=3\}$
\item Klein space: $\mathbb{K}=\{V_2\subset\mathbb{T}\mid \dim V_2=2\}$
\item Variety of complete flags: $\op{GL}(\mathbb{T})/B=\{(V_1,V_2,V_3)\mid V_1\subset V_2\subset V_3\subset\mathbb{T},\ \ \dim V_1=1,\ \dim V_2=2,\ \dim V_3=3\}$
\end{itemize}

The flag manifolds are obtained as quotients from each other in the following diagram, showing the Dynkin type of each flag manifold
\begin{equation}\label{bigpicture}
\xymatrix{
& \underset{\Axxx}{\op{GL}(\mathbb{T})/B} \ar[dl]_{(V_1,V_2,V_3)\mapsto (V_2,V_3)}\ar[rr]^{(V_1,V_2,V_2)\mapsto (V_1,V_2)}\ar'[d]^{(V_1,V_2,V_3)\mapsto (V_1,V_3)}[dd] & & \underset{\Axxo}{\mathbb{PS}'} \ar[dd]^{(V_1,V_2)\mapsto V_1}\ar[dl]^{(V_1,V_2)\mapsto V_2}\\
\underset{\Aoxx}{\mathbb{PS}} \ar[rr]_{(V_2,V_3)\mapsto V_2}\ar[dd]_{(V_2,V_3)\mapsto V_3} & & \underset{\Aoxo}{\mathbb{K}}\ar[dd] &\\
& \underset{\Axox}{\mathbb{PN}} \ar'[r]_{(V_1,V_3)\mapsto V_1}[rr]\ar[dl]^{(V_1,V_3)\mapsto V_3} & & \underset{\Axoo}{\mathbb{PT}}\ar[dl]
\\
\underset{\Aoox}{\mathbb{PT}^*} \ar[rr] & & \underset{\Aooo}{\{\star\}}  &
}
\end{equation}

The space $\mathbb{K}$ is isomorphic as a $\op{GL}(\mathbb{T})$ manifold to the Klein quadric in $\mathbb{P}(\wedge^2\mathbb{T})$.  This identification associates to the $2$-plane $V_2\subset\mathbb{T}$ (defining a point of $\mathbb{K}$) the simple $2$-vector obtained by wedging together a pair of generators of $V_2$.  The Klein quadric is the locus of $X\wedge X=0$.  This carries a conformal structure obtained by declaring two points $X,Y\in\mathbb{K}$ to be null-related if and only if $X\wedge Y=0$ in $\mathbb{P}(\wedge^2\mathbb{T})$.

Since $\mathbb{K}$ is an indefinite $(3,3)$ quadric, it is ruled by two distinct families of completely isotropic $2$-planes.  Distinguish one family of planes as $\alpha$-planes and the other as $\beta$-planes.  The space $\mathbb{PT}$ is identified with the space of $\alpha$-planes and $\mathbb{PT}^*$ is identified with the space of $\beta$-planes.  An $\alpha$-plane $A$ and $\beta$-plane $B$ intersect (are {\em incident}) if and only if the one-dimensional subspace of $\mathbb{T}$ corresponding to $A\in\mathbb{PT}$ is contained in the three-dimensional subspace of $\mathbb{T}$ corresponding to $B\in\mathbb{PT}^*$.  Thus $\mathbb{PN}$ is precisely the set of $\alpha$-planes and $\beta$-planes that intersect.  Their intersection is always a null geodesic of $\mathbb{K}$.

This construction identifies $\mathbb{PN}$ naturally with the subset of $\mathbb{PT}\times\mathbb{PT}^*$ consisiting of $A\in\mathbb{PT}$ and $B\in\mathbb{PT}^*$ that are incident.  Each of the projections $\pi_1:\mathbb{PN}\to\mathbb{PT}$ and $\pi_3:\mathbb{PN}\to\mathbb{PT}^*$ have rank three, with $\mathbb{P}^2$ fibers.

\subsection{Structure at infinity}\label{StructureAtInfinity}
Fix a point $I\in\mathbb{K}$.  Let $\scri\subset\mathbb{K}$ (or ``scri'') be the null cone through $I$, minus the vertex.  Let $\mathbb{P}\scri\subset\mathbb{PN}$ be the space of all null geodesics through $I$.  In terms of flags,
$$\mathbb{P}\scri = \{ (V_1,V_3)\mid V_1\subset I\subset V_3\}.$$
Note that $\mathbb{P}\scri$ is fibered over the image of $I$ in $\mathbb{PT}$ and in $\mathbb{PT}^*$, so $\mathbb{P}\scri$ is naturally a product $\mathbb{P}^1\times\mathbb{P}^1$, and this product structure in preserved under the stabilizer of $I$ in $\op{GL}(\mathbb{T})$.  Under the double fibration of $\op{GL}(\mathbb{T})/B$ \eqref{bigpicture}, $\scri$ is precisely the set of points in $\mathbb{K}$ incident with $\mathbb{P}\scri$. Scri is a trivial real line bundle over $\mathbb{P}\scri=\mathbb{P}^1\times\mathbb{P}^1$.  Introduce projective coordinates $(\pi',\pi)$ on $\mathbb{P}^1\times\mathbb{P}^1$.  An affine parameter $u$ on the line bundle $\scri\to\mathbb{P}\scri$ then gives a coordinate system $(u,\pi',\pi)$ on $\scri$.  The preimage of a $\pi'=\pi'_0$ circle on $\mathbb{P}\scri$ is an $\alpha$-plane lying on $\scri$ and the preimage of a circle $\pi=\pi_0$ is a $\beta$-plane lying on $\scri$.  Every $\alpha$ and $\beta$ plane of $\scri$ is obtained in this manner.

Generically, a null geodesic in $\mathbb{K}$ intersects $\scri$ at a unique point.  This induces a fibration $\mathbb{PN}\dashrightarrow\scri$ in the category of rational maps.  More specifically, let $U$ be the set of flags $(V_1,V_3)\in\mathbb{PN}$ such that $V_1\not\subset I$ and $I\not\subset V_3$.  Then there exists a unique $J\in\scri$ such that $V_1\subset J\subset V_3$.  This $J$ defines a fibration $U\to\scri$ with $\mathbb{P}^1\times\mathbb{P}^1$ fibers.

\section{Shearfree congruences}
A smoothly parametrized family of null geodesics that foliate an open subset of $\mathbb{K}$ in called a {\em null geodesic congruence}.  Fix a representative metric for the conformal structure on $\mathbb{K}$.  Let $k$ be a nonvanishing affinely parametrized vector field that is everywhere tangent to a null geodesic congruence.  Define the {\em umbral bundle} to be the bundle $E=k^\perp/k$ consisting of all vectors that are perpendicular to $k$ modulo $k$ itself.  The metric on $\mathbb{K}$ descends to a metric of signature $(1,1)$ on $E$.  Since $k$ is null and affinely parametrized, the linear transformation $X\mapsto \nabla_X k$ descends to a linear transformation on $(\nabla k)_E:E\to E$.  The {\em shear tensor} of the congruence is the tracefree symmetric part of $(\nabla k)_E$.  If the shear tensor vanshes, then the congruence is said to be {\em shearfree}.  The property of being shearfree is invariant under conformal changes to the metric.

\begin{lemma}
Let $m$ and $m'$ be two (local) null sections of $E$ that are linearly independent.  The shear vanishes if and only if $g_E(\nabla_m k,m) = g_E(\nabla_{m'}k,m')=0$.  Equivalently, the shear vanishes if and only if the distributions spanned by $k,m$ and by $k,m'$ are integrable in the sense of Frobenius.
\end{lemma}

The quotient by the null geodesics in the congruence defines a submanifold of $\mathbb{PN}$.  The condition that it be shearfree is equivalent to $m$ and $m'$ descending under the quotient to elements of the distributions $\ker d\pi_1$ and $\ker d\pi_3$ on $\mathbb{PN}$.  A shearfree congruence is thus identified locally with a three-dimensional embedded submanifold $K$ of $\mathbb{PN}$ such that $\ker d\pi_1\cap TK$ and $\ker d\pi_3\cap TK$ are both one-dimensional distributions.  That is, $K$ is a $3$-dimensional real CR submanifold of $\mathbb{PN}$.  Conversely, such a submanifold gives rise naturally to a shearfree congruence in an open subset of $\mathbb{K}$.

Confining attention to congruences that are not tangent to $\scri$ at any point, the submanifold $K$ defines a local section over the fibration $\mathbb{PN}\dashrightarrow\scri$ denoted by $\kappa$.  Let $\kappa=(\kappa_1,\kappa_3)$ where $\kappa_1=\pi_1\kappa$ and $\kappa_3=\pi_3\kappa$ are the components of $\kappa$ under the projections $\pi_1:\mathbb{PN}\to\mathbb{PT}$ and $\pi_3:\mathbb{PN}\to\mathbb{PT}^*$.  The condition that $\kappa$ be a shearfree congruence is that $\kappa_1$ and $\kappa_3$ both have rank two everywhere.

\subsection{Shearfree congruences at $\scri$}
As in Section \ref{StructureAtInfinity}, the projective coordinates $\pi$ and $\pi'$ on $\scri$ label, respectively, the $\beta$ and $\alpha$ planes on $\scri$, so that for a fixed $\pi_0$ and $\pi'_0$, $\pi^{-1}(\pi_0)$ and $(\pi')^{-1}(\pi'_0)$ are a $\beta$ and an $\alpha$ plane, respectively.  Fix a $\beta$-plane, $B=\pi^{-1}(\pi_0)\subset\scri$.  Then the restriction of $\kappa_1$ to $B$ associates to each point $x\in B$ an $\alpha$-plane incident with $x$. The $\alpha$-plane $\kappa_1(x)$ intersects $B$ in a straight line through $x$ (which is a null geodesic, since $B$ is isotropic).  

For any $x\in B$, the lines in $\mathbb{PT}$ that join $\kappa_1(x)$ to points of $I$ determine a straight line on $\scri$ as well.  The fiber of $\kappa_1$ over $\kappa_1(x)$ must be contained in this line.  The fiber cannot degenerate to a point since $\kappa_1$ has rank two.

These two lines determined by $\kappa_1(x)$ are the same line.  Indeed, they both lie on the plane containing $\kappa_1(x)$ and $I$, which is dual to the element of $\mathbb{PT}^*$ that defines the $\beta$-plane $B$.  Moreover, they both lie on the $\alpha$-plane containing $\kappa_1(x)$.  But these conditions determine a unique $\mathbb{P}^1$.  Therefore $\kappa_1|_B$ is a Burgers' function.

In the global case, the fiber of each $\kappa_i$ is an open subset of a line on $\scri$ (a $\mathbb{P}^1$), and so by completeness must be the entire line.  Since the fibers are compact, the quotient of $\scri$ by the fibers is a smooth surface, and so each $\kappa_i$ factors through the quotient with an embedding into $\mathbb{PT}$ and $\mathbb{PT}^*$, respectively.

\subsection{Cauchy problem}
Let $S$ be a section of the line bundle $\scri\to\mathbb{P}^1\times\mathbb{P}^1$.  Let $\kappa_S:S\to\mathbb{PN}$ be a smooth section of the fiber product $S\times_{\scri}\mathbb{PN}\to S$, and assume that $\kappa_S$ is transverse to both of the fibrations $\mathbb{PN}\to\mathbb{PT}$ and $\mathbb{PN}\to\mathbb{PT}^*$.  The Cauchy problem is to find a shearfree congruence that agrees with $\kappa_S$ when restricted to $S$.  

The problem of local existence is twofold:
\begin{itemize}
\item Find an open neighborhood $V$ of $S$ on $\scri$ and a section $\kappa:V\to \mathbb{PN}$ such that $\kappa|_S=\kappa_S$ and $\kappa_1,\kappa_2$ are rank two.
\item Find an open neightborhood $U$ of $S$ in $\mathbb{K}$ such that there is a shearfree congruence in $U$ that induces the data $\kappa$ at the boundary.
\end{itemize}

The first of these is assured by the local solvability of Burgers' equation (Lemma \ref{BurgersCauchy}) for smooth initial data and compactness of the domain.  

The second is then the problem of finding a foliation in a sufficiently small open set $U$ by vectors $k$ subject to an initial condition on the set $V$.  There is a neighborhood $U_x$ of any point $x\in S$ in which there is a shearfree congruence of the required kind, by the implicit function theorem.  Since $S$ is compact, it is covered by a finite number of such sets.  Passing to a refinement if necessary then ensures that the foliation in each $U_x$ extends to a foliation of their union.  This proves Theorem \ref{localexistence}.

\section{Curved case}
Now let $M$ be a manifold with a $(2,2)$ conformal structure.  The conformal structure is specified by following data of a Cartan connection (see \cite{Sharpe}):
\begin{itemize}
\item A principal $H$-bundle $\mathscr{G}\to M$, with $H=P(\Aoxo)$, and right action $R:H\to \op{Diff}(\mathscr{G})$.
\item A connection form $\omega:T\mathscr{G}\to \mathfrak{gl}(\mathbb{T})$ that is an isomorphism of each tangent space equivariantly interwining the right action with the adjoint action of $H$, and such that $\omega(R_*(X))=X$ for all $X\in\mathfrak{h}$.
\end{itemize}
A curvature normalization on $\omega$ condition uniquely determines these data.  In the flat case, $\mathscr{G}$ is a principal homogeneous space for $\op{GL}(\mathbb{T})$ and $\omega$ is the associated Maurer--Cartan form.

The following quotient bundles corresponding to parabolic subgroups of $P(\Aoxo)$ play a role:
\begin{itemize}
\item $\mathscr{G}/P(\Aoxx)=\mathbb{P}\mathscr{S}$ is the projective spin bundle
\item $\mathscr{G}/P(\Axxo)=\mathbb{P}\mathscr{S}'$ is the projective primed spin bundle
\item $\mathscr{G}/P(\Axxx)=\mathbb{P}\mathscr{S}\times_M\mathbb{P}\mathscr{S}'=:\mathbb{P}\mathscr{N}$ is the bundle of null directions
\end{itemize}

The twistor distribution on $\mathscr{G}$ is the codimension $3$ distribution on $\mathscr{G}$ defined by $\mathbf{D}=\omega^{-1}(\mathfrak{p}(\Axoo))\subset T\mathscr{G}$.  The dual twistor distribution is the distribution $\mathbf{D}'=\omega^{-1}(\mathfrak{p}(\Aoox))$.  An $\alpha$-surface in $M$ is the projection of an integral surface of $\mathbf{D}$ and a $\beta$-surface is the projection of an integral surface of $\mathbf{D}'$.  In the flat case, these coincide with $\alpha$ and $\beta$ planes, respectively.  In general, however, the distributions $\mathbf{D}$ and $\mathbf{D}'$ are not integrable: a necessary and sufficient condition for integrability of each is the vanishing of a corresponding $\op{SO}(2,2)$ irreducible component of the Weyl tensor.  

Each of $\mathbf{D}$ and $\mathbf{D}'$ descend to $3$-dimensional distributions $D$ and $D'$ on $\mathbb{P}\mathscr{N}$, the bundle of null directions.  Their intersection $D\cap D'$  defines the null geodesic spray on this bundle.


The {\em shear} is encoded in the pair of integrability obstructions for $D$ and for $D'$, respectively.  Specifically, let $V$ be a nonzero representative of $\mathbf{D}\cap\mathbf{D}'$ and let $M$ and $M'$ complete $V$ to a basis of $D$ and of $D'$, respectively.   The shear is then given by the pair three-vectors $\Sigma = [V,M]\wedge V\wedge M, \Sigma' = [V,M']\wedge V\wedge M'$, which are well-defined up to an overall nonzero scale.

A congruence of null geodesics is a submanifold $K$ of the bundle $\mathbb{P}\mathscr{N}$ that is tangent to $D\cap D'$ and such that the projection $K\to M$ has rank $4$.  Then each of $D\cap TK$ and $D'\cap TK$ is a two-dimensional distribution on $K$.  The {\em shear} of the congruence is the integrability obstrubtion to each of these distributions.  Specifically, let $V$ be a nonzero representative of $\mathbf{D}\cap\mathbf{D}'$ and let $M$ and $M'$ complete $V$ to a basis of $D\cap TK$ and of $D'\cap TK$, respectively.   The shear is then given by the pair three-vectors $\Sigma = [V,M]\wedge V\wedge M, \Sigma' = [V,M']\wedge V\wedge M'$, which are well-defined up to an overall nonzero scale.  In particular, a congruence is {\em shearfree} if the shear vanishes identically on the section.

Let $i:N\xrightarrow{\subset} M$ be an embedded hypersurface.  Then $i^{-1}\mathbb{P}\mathscr{N}$, $i^{-1}\mathbb{P}\mathscr{S}$, and $i^{-1}\mathbb{P}\mathscr{S}'$ are, respectively, the pullback of the bundle of null directions, the pullback of the projective spin bundle, and the pullback of the projective spin bundle along the inclusion of $N$ in $M$. The distributions $D_N=D\cap Ti^{-1}\mathbb{P}\mathscr{N}$ and $D_N'=D'\cap Ti^{-1}\mathbf{P}\mathscr{N}$ each defines a distribution on $T\mathbb{P}\mathscr{N}_N$.  Each of these distributions is integrable.  Indeed, the distribution $D_N$ descends to a vector field on $i^{-1}\mathbb{P}\mathscr{S}$ and contains the vertical distribution for the fibration $i^{-1}\mathbb{P}\mathscr{N}\to\mathbb{P}\mathscr{S}$; similarly for $D'_N$.  So on the open subset of $\mathbb{P}\mathscr{N}_N$ obtained by deleting the null directions that are tangent to $N$, $D_N$ and $D'_N$ define a real CR structure.

The Levi form of this real CR structure is non-degenerate.  Indeed, let $E_{ij}$ be the matrix in $\mathfrak{gl}(\mathbb{T})$ with a $1$ in the $ij$ position, and zeros everywhere else.  Then the vector field $X=\omega^{-1}E_{43}$ lies in $D_N$ and is vertical.  Any vector field $Y\in D_N'$ has the form
$$Y\equiv a\omega^{-1}E_{31} + b\omega^{-1}E_{32} \pmod{\omega^{-1}\mathfrak{p}(\Aoxo)}.$$
If $Y$ is not vertical, then $a\not=0$, since $\omega^{-1}E_{32}$ generates $D_N\cap D_N'$ but we have deleted those directions.  By equivariance of the connection under Lie differentiation in the vertical direction $X$, $a$ is the $E_{41}$ component of $[X,Y]$, which is nonzero by assumption.

A congruence of null geodesics $K$ in a neighborhood of $N$ in $M$ is {\em asymptotically shearfree}\footnote{Although in the physics literature, often the implication is that the hypersurface $N$ be at infinity, we require no such restriction here but retain the terminology.} if the shear of $K$ vanishes identically on $N$.  An asymptotically shearfree congruence is determined by a section $\kappa : N \to \mathbb{P}\mathscr{N}_N$ onto a codimension $2$ submanifold $K$ such that $D_N\cap TK$ and $D'_N\cap TK$ both have rank one everywhere.  Thus $K$ is a real CR submanifold of $\mathbb{P}\mathscr{N}_N$.  Summarizing,

\begin{theorem}
An asymptotically shearfree congruence defines a codimension $2$ real CR submanifold of the bundle $\mathbb{P}\mathscr{N}_N\to N$ such that the projection onto $N$ is rank $3$.  Conversely, an asymptotically shearfree congruence is determined in a neighborhood of each point of $N$ by this data.
\end{theorem}

Let $A\subset M$ be an $\alpha$-surface.  Let $\mathscr{A}\subset\mathscr{G}$ be the corresponding maximal integral manifold of $\mathbf{D}$.  Then $\mathscr{A}$ is a $P(\Axxo)$-bundle over $A$.  The distribution $\mathbf{D}\cap\mathbf{D}'$ restricts to $\mathscr{A}$.  This distribution is invariant under the action of $P(\Axxx)$, and so descends to a one-dimensional distribution on $\mathbb{P}\mathscr{S}'_A:=\mathscr{A}/P(\Axxx)$, the pullback of the projective primed spin bundle to $A$.  Since $\mathbb{P}\mathscr{S}'_A\cong\mathbb{P}T^*A$, this direction field is a spray on $A$.  Under this isomorphism, it is the geodesic spray restricted to $A$, and therefore defines a projective structure on $A$.  Similarly, if $B$ is a $\beta$-surface then the pullback of the projective spin bundle $\mathbb{P}\mathscr{S}_B$ is identified with the cotangent bundle of and the geodesic spray restricts to a Legendrian vector field on it.

A null hypersurface $N$ is called {\em shearfree} if it is foliated by $\alpha$-surfaces and by $\beta$-surfaces.  An asymptotically shearfree congruence defines a submanifold $K$ of the bundle $\pi:\mathbb{P}\mathscr{N}_N\to N$ on which $\pi$ has rank $3$.  If $A$ is an $\alpha$-surface of $N$, then factoring $K$ by the quotient to $\mathbb{P}\mathscr{S}'_A$ gives a submanifold $K_A$ of $\mathbb{P}\mathscr{S}'_A$ whose projection onto $A$ has rank two.  Moreover, $K_A$ is foliated by the leaves of the Legendrian vector field $D'$.  Thus $K_A$ is a Burgers' surface.  Likewise, each $\beta$-surface of $N$ gives rise to a Burgers' surface.

\begin{theorem}
An asymptotically shearfree congruence in a neighborhood of each point of a shearfree null hypersurface $N$ is determined by a one-parameter family of generalized Burgers' surfaces $K_A$ over the $\alpha$-surfaces that foliate $N$, and a one-parameter family of Burgers' surfaces $K_B$ over the $\beta$-surfaces that foliate $N$.  Locally, this is determined by a solution $(L,M)$ to the pair of inhomogeneous Burgers' equations
$$L_u+LL_x = \sigma(u,x,y,L),\qquad M_u+MM_y = \tilde{\sigma}(u,x,y,M)$$
where the coordinate $y$ labels $\beta$-surfaces of the foliation, the coordinate $x$ labels the $\alpha$-surfaces of the foliation, and $\sigma$ and $\tilde{\sigma}$ are functions defining the null geodesic sprays in each of the $\alpha$ and $\beta$-planes, respectively.  Morever, $\sigma_{LLLL}=\tilde{\sigma}_{MMMM}=0$.
\end{theorem}

\bibliography{klein}{}
\bibliographystyle{plain}
\end{document}